\documentclass{amsart}
\usepackage{amsmath, amsthm, amsfonts, mathrsfs}
\usepackage[
pdftitle={..:: Group gradings on upper block triangular matrices ::..},
pdfdisplaydoctitle=true,
hidelinks,
]{hyperref}

\newcounter{contador}
\newtheorem{Thm}[contador]{Theorem}

\newtheorem{Cor}[contador]{Corollary}
\newtheorem{Lemma}[contador]{Lemma}
\theoremstyle{definition}

\theoremstyle{remark}

\title[Group gradings on upper block triangular matrices]{Group gradings on upper block triangular matrices\footnote{This is a post-peer-review, pre-copyedit version of an article published in Archiv der Mathematik. The final authenticated version is available online at: http://dx.doi.org/10.1007/s00013-017-1134-0}}
\author{Felipe Yukihide Yasumura}
\address{
Department of Mathematics, State University of Maring\'a, Maring\'a, PR, Brazil.}
\email{felipeyukihide@gmail.com}
\thanks{This work was supported by Fapesp, grant no. 2013/22.802-1 and grant no. 2017/11.018-9}

\begin{document}
\begin{abstract}
	It was proved by Valenti and Zaicev, in 2011, that, if $G$ is an abelian group and $K$ is an algebraically closed field of characteristic zero, then any $G$-grading on the algebra of upper block triangular matrices over $K$ is isomorphic to a tensor product $M_n(K)\otimes UT(n_1,n_2,\ldots,n_d)$, where $UT(n_1,n_2,\ldots,n_d)$ is endowed with an elementary grading and $M_n(K)$ is provided with a division grading.

	In this manuscript, we prove the validity of the same result for a non necessarily commutative group and over an adequate field (characteristic either zero or large enough), not necessarily algebraically closed.
\end{abstract}

\maketitle

\section{Introduction}
Algebras with additional structure are deeply studied nowadays, and in particular, the graded algebras was intensily investigated mainly after the works of Kemer \cite{kemer}, showing the importance of $\mathbb{Z}_2$-graded algebras in the study of algebras with polynomial identities. These algebras constitutes a natural generalization of polynomial algebras in the commutative case. They are also related with supersymmetries in Physics. An interesting question concerning gradings on algebras is classifying all possible gradings on a given algebra. For simple associative, Lie and Jordan algebras, the classification is essentially complete (see the book \cite{EldKoch} for a complete reference in the subject). There exists many other algebras whose gradings was computed or partially computed.

In this manuscript, we are interested in studying a non-simple algebra, namely the upper block triangular matrices. These algebras are defined in the following way. Let $n_1,n_2,\ldots,n_t\in\mathbb{N}$ be any integers, then set
$$
UT(n_1,n_2,\ldots,n_t)=\left(\begin{array}{cccc}A_{11}&A_{12}&\ldots&A_{1t}\\0&A_{22}&\ldots&A_{2t}\\\vdots&\ddots&\ddots&\vdots\\0&\ldots&0&A_{tt}\end{array}\right),
$$
where each $A_{ij}$, for $1\le i\le j\le t$, is a $n_i\times n_j$ matrix with entries in the field $K$. The Jacobson Radical $J$ of $UT(n_1,n_2,\ldots,n_t)$ is the set of all elements such that all $A_{ii}$ are zero, for $i=1,2\ldots,t$. The upper triangular matrices is a particular case of upper block triangular matrices, if we consider $UT(1,1,\ldots,1)$. The matrix algebras can also be obtained if we put $t=1$.

In 2003, Valenti and Zaicev proved that any group grading on the algebra of upper triangular matrices over an algebraically closed field of characteristic zero, where the grading group is abelian, is elementary, up to a graded isomorphism \cite{VZ2003}. In 2007, the same authors proved the same theorem, but for arbitrary field and any group \cite{VZ2007}; and in the same paper the authors conjectured the classification of the group gradings over the algebra of upper block triangular matrices. But in 2011, Valenti and Zaicev solved this question, proving the validity of their conjecture for an algebraically closed field of characteristic zero and the grading group commutative and finite \cite{VZ2011}.

Following the sequence, in this manuscript, we describe the group gradings on the upper block triangular matrices, proving the conjecture of Valenti and Zaicev for arbitrary field of characteristic zero (or the characteristic greater than the dimension of the algebra) and a group not necessarily commutative, nor finite.

We recall that the upper block triangular matrices, in the ungraded sense, are related to the so called minimal varieties (see \cite{GS2005} and the references therein). The classification of the elementary gradings on the upper block triangular matrices was studied in \cite{BD2013}. The graded polynomial identities for the elementary gradings on the upper block triangular matrices was dealt in \cite{CM2014,MS2016}. Also, in \cite{VinSpi2014} the authors addressed the question of when the knowledge of the graded polynomial identities for a certain grading on the upper block triangular matrices completely determines the grading.

\section{Notations and preliminaries}
We fix a group $G$ with multiplicative notation and an arbitrary field $K$.

\paragraph{Graded algebras.} Let $A$ be any algebra (associative or not) and $G$ any group. We say that $A$ is a \textit{$G$-graded algebra} (or $A$ is equiped with a $G$-grading) if there exists a vector space decomposition $A=\bigoplus_{g\in G}A_g$ (where some of the $A_g$ can be zero) satisfying $A_gA_h\subset A_{gh}$ for all $g,h\in G$. We call the elements in $\cup_{g\in G} A_g$ homogeneous, and we say that $x$ has degree $g$ if $x\in A_g$, denoted $\deg x=g$. A \textit{graded division algebra} is a graded algebra $A$ such that every non-zero homogeneous element of $A$ has an inverse in $A$.

\paragraph{Gradings on matrix algebras.} We say that a $G$-grading on $M_n$ is \emph{elementary} if there exists a sequence $(g_1,g_2,\ldots,g_n)\in G^n$ such that every matrix unit $e_{ij}\in U$ is homogeneous of degree $g_ig_j^{-1}$. If $B$ is another $G$-graded algebra, then we can furnish a $G$-grading on $M_n\otimes_K B$ if we put
\begin{equation}\label{eq_degree}
\deg e_{ij}\otimes b=g_i\deg b g_j^{-1},
\end{equation}
for all homogeneous $b\in B$.

We canonically identify $M_n\otimes M_m=M_{nm}$ via Kronecker product. It is well known that the graded version of the Density Theorem holds valid. That is, given $M_n$ endowed with a $G$-grading, we can find $M_r(K)$ equipped with an elementary grading given by a sequence $\mu$, and a graded division algebra $D=M_s(K)$ such that $M_n\cong M_r\otimes D$, where the grading on the tensor product is given by \eqref{eq_degree}. In this case, we denote such grading by $(M_n,D,\mu)$.

\paragraph{Notations: upper block-triangular matrices.} Denote by $J$ the Jacobson Radical of $U=UT(n_1,n_2,\ldots,n_t)$. Denote also by $M_{ij}$ the block of matrices, so that we can write (as vector spaces) $U=\bigoplus_{1\le i\le j\le t}M_{ij}$. Thus, in this notation $J=\bigoplus_{i<j}M_{ij}$. Note that each $M_{ii}$ is isomorphic to the $n_i\times n_i$ matrix algebra, and we can see $M_{ii}$ as a subalgebra of $U$. Let $E_i\in M_{ii}$ be its identity matrix.

It may happen that $U$ is a graded subalgebra of some $(M_n,D,\mu)$. This will happen if and only if $r$ divides each $n_i$, where $D$ consists of $r\times r$ matrices. We denote by $(U,D,\mu)$ the $G$-grading on $U$ induced from $(M_n,D,\mu)$.

\paragraph{Valenti-Zaicev Conjecture.} In \cite{VZ2007}, Valenti and Zaicev conjectured that every grading on $U$ is graded isomorphic to $UT(n_1',n_2',\ldots,n_t')\otimes M_n(K)$, where $M_n(K)$ is provided with a division grading and $UT(n_1',n_2',\ldots,n_t')$ is endowed with an elementary grading. This was proved to be true, if the base field is algebraically closed of characteristic zero and the group is finite and abelian \cite{VZ2011}.

\paragraph{Graded modules.} Let $A$ be a $G$-graded algebra and $V$ a vector space that is an $A$-module. Suppose that we have a decomposition $V=\bigoplus_{g\in G}V_g$ into subspaces (in this case, $V$ has a vector space grading). We say that $V$ is a \textit{graded $A$-module} if $V_gA_h\subset V_{gh}$, for all $g,h\in G$.

If $V=\bigoplus_{g\in G}V_g$ is a graded vector space, and given $h\in G$, we define $V^{[h]}$ as the graded vector space with decomposition $V^{[h]}=\bigoplus_{g\in G}V_g^{[h]}$, where $V_{gh}^{[h]}=V_{g}$. Similarly we define the graded vector space ${}^{[h]}\!V$. Note that if $A$ is a $G$-graded algebra, then $A$ itself is a $G$-graded $A$-module.

For the special case where $D$ is a $G$-graded division algebra, the structure of $D$-modules are well known (see, for instance, \cite[Chapter 2, page 29]{EldKoch}). If $V$ is a $G$-graded $D$-module, then $V=\bigoplus V_i$, where each $V_i={}^{[g_i]}\!D$, for some $g_i\in G$. In other words, every graded $D$-module is free.

\section{Group gradings on the upper block triangular matrices}

We start proving that some subspaces are graded:
\begin{Lemma}
	If $J$ is graded, then all $M_{ij}$ are graded subspaces.
\end{Lemma}
\begin{proof}
	Recall that the annihilator (left, right or two-sided) of a graded subset is again graded. Then $R:=\text{Ann}_U^r(J)=\bigoplus_{j=1}^t M_{1j}$ (the right annihilator of $J$) is graded. 

It is well known that the unity of an unital associative graded algebra is always homogeneous. Exactly the same argument can be used to prove the following: if an associative algebra has a left unit, then there exists a homogeneous left unity in the algebra. Note that $R$ has a left unity (the identity matrix $E_1\in M_{11}$), hence it must admit a homogeneous left unit, say $u_1$. Clearly $u_1^2=u_1$, hence $u_1$ is diagonalizable; moreover, the diagonal form of $u_1$ is exactly $E_1$. So, after applying an isomorphism, we can assume $E_1$ homogeneous.

	Now, since $(1-E_1)U\cong UT(n_2,n_3,\ldots,n_t)$ we can proceed by induction. Moreover, if $i<j$ and $E_i$ and $E_j$ are the identity matrices of $M_{ii}$ and $M_{jj}$, respectively, then $M_{ij}=E_iUE_j$ is a graded subspace.
\end{proof}

So we can assume every matrix subalgebra $M_{ii}$ graded. It follows that every $M_{ii}\cong\mathcal{M}_{i}\otimes D_i$, where $\mathcal{M}_{i}$ is a $p_i\times p_i$ matrix algebra equiped with an elementary grading given by $(g_1^{(i)},\ldots,g_{p_i}^{(i)})$, and $D_i$ is a graded division algebra, where the grading on $\mathcal{M}_{i}\otimes D_i$ is induced by (\ref{eq_degree}). Since we identify $M_{ii}=\mathcal{M}_{i}\otimes D_i$, we also (equivalently) identify $M_{ii}=\mathcal{M}_{i}(D_i)$, the $p_i\times p_i$ matrix algebra with coefficients in $D_i$. We denote the elements of $M_{ii}$ as linear combination of $m\otimes d$, where $m\in\mathcal{M}_{i}$, and $d\in D_i$. As mentioned before, we assume that each $M_{ii}$ is a subalgebra of $U$. Moreover, under these identifications, each $M_{ij}$ is a (graded) $(M_{ii},M_{jj})$-bimodule; and $U$ is a (graded) $(M_{ii},M_{jj})$-bimodule as well.

It is well known that every automorphism of a matrix algebra is inner, hence we can find an invertible matrix $A_i$ such that $A_iM_{ii}A_i^{-1}=\mathcal{M}_{i}\otimes D_i$, where the grading on $\mathcal{M}_{i}\otimes D_i$ is given by \eqref{eq_degree}. Taking the block-diagonal matrix $A'=\text{diag}(A_1,A_2,\ldots,A_t)$, we obtain an automorphism of $U$ such that every $M_{ii}=\mathcal{M}_{i}\otimes D_i$.

Denote the matrix units of each $\mathcal{M}_{k}$ by $e_{ij}^{(k)}$. Given $e_{ii}^{(r)}\in\mathcal{M}_{r}$, $e_{jj}^{(s)}\in\mathcal{M}_{s}$, let
\begin{equation}\label{eq1}
V=V_{ij}^{(r,s)}=(e_{ii}^{(r)}\otimes1)U(e_{jj}^{(s)}\otimes1).
\end{equation}
$V$ is a graded subspace of $U$; moreover, $V$ is a $(D_r,D_s)$-bimodule via
$$
d_1\ast v\ast d_2=(e_{ii}^{(r)}\otimes d_1)v(e_{jj}^{(s)}\otimes d_2),\quad d_1\in D_r,d_2\in D_s,v\in V.
$$
Note that $e_{ii}^{(r)}\otimes D_r$ is a graded subalgebra of $U$, and $e_{ii}^{(r)}\otimes D_r\cong{}^{[g_i^{(r)}]}D_r^{[(g_i^{(r)})^{-1}]}$. Similarly, $e_{jj}^{(s)}\otimes D_s\cong{}^{[g_j^{(s)}]}D_s^{[(g_j^{(s)})^{-1}]}$. Thus $V$ is a graded $(e_{ii}^{(r)}\otimes D_r,e_{jj}^{(s)}\otimes D_s)$-bimodule.
\begin{Lemma}
In the notation above, for any nonzero homogeneous $v\in V$, we have $V=D_r\ast v=v\ast D_s$. Moreover, if $h=\deg v$, then there exists a weak isomorphism $\psi_{rs}=\psi_{r,s,i,j}:D_r\to D_s$ such that $d\ast v=v\ast\psi_{rs}(d)$, for all $d\in D_r$, and
\begin{equation}\label{weakiso}
g_j^{(s)}\deg_{D_s}\psi_{rs}(d)\left(g_j^{(s)}\right)^{-1}=g_i^{(r)}h(\deg_{D_r}d)h^{-1}\left(g_i^{(r)}\right)^{-1},
\end{equation}
for any non-zero homogeneous $d\in D_r$.
\end{Lemma}
\begin{proof}
If $e_{ii}^{(r)}\otimes D_r$ consists of $n_r'\times n_r'$ matrices and $e_{jj}^{(s)}\otimes D_s$ is $n_s'\times n_s'$ matrices then $V$ is $n_r'\times n_s'$ matrices. From the structure of graded modules over graded division algebras, we obtain $n_r'n_s'=k_1n_r^{\prime2}=k_2n_s^{\prime2}$, for some $k_1,k_2\in\mathbb{N}$. This is possible only if $n_r'=n_s'$, so $\dim_{D_r}V=\dim_{D_s}V=1$. Hence, given a non-zero homogeneous $v\in V$ of degree $h\in G$, we have $V=D_r\ast v=v\ast D_s$. As a consequence, for any $x\in e_{ii}^{(r)}\otimes D_r$, there exists $y\in e_{jj}^{(s)}\otimes D_s$ such that $xv=vy$; in particular, if $x$ is homogeneous, then $y$ is homogeneous as well and $\deg x=h(\deg y)h^{-1}$. Let $T:x\in e_{ii}^{(r)}\otimes D_r\mapsto y\in e_{jj}^{(s)}\otimes D_s$. Clearly $T$ is a linear map. Also, for each homogeneous $x\in e_{ii}^{(r)}\otimes D_r$, one has $\deg T(x)=h^{-1}(\deg x)h$. Futhermore, $vT(x_1x_2)=x_1x_2v=x_1vT(x_2)=vT(x_1)T(x_2)$. Since $D_r$ is a graded division algebra, one obtains $T(x_1x_2)=T(x_1)T(x_2)$, which means that $T$ is a homomorphism of algebras. Thus, $T$ is a weak isomorphism between $e_{ii}^{(r)}\otimes D_r$ and $e_{jj}^{(s)}\otimes D_s$. Finally, we can define $\psi_{rs}$ by the composition of weak isomorphisms $\psi_{rs}:D_r\cong e_{ii}^{(r)}\otimes D_r\stackrel{T}{\to}e_{jj}^{(s)}\otimes D_s\cong D_s$.
\end{proof}

Now, for each $r$, we set (as in \eqref{eq1})
$$
V^{r,r+1}=V_{p_r,1}^{(r,r+1)}.
$$
Let $v^{r,r+1}\in V^{r,r+1}$ be a nonzero homogeneous. Denote $\psi_{r,r+1}:D_r\to D_{r+1}$ the respective weak isomorphism as in the previous lemma. For each $r<s$, let
\begin{align*}
&\psi_{r,s}=\psi_{s-1,s}\circ\cdots\circ\psi_{r,r+1}:D_r\to D_s,\\%
&v^{rs}=v^{r,r+1}(e_{1,p_{r+1}}^{(r+1)}\otimes1)v^{r+1,r+2}(e_{1,p_{r+2}}^{(r+2)}\otimes1)\cdots v^{s-1,s}.
\end{align*}
\noindent\textbf{Claim.} $v^{rs}\ne0$ is homogeneous, $\psi_{rs}$ is a weak isomorphism and
\begin{equation}\label{eqcom}
d\ast v^{rs}=v^{rs}\ast\psi_{rs}(d),
\end{equation}
for each homogeneous $d\in D_r$.

Indeed, the $\psi_{rs}$ is a composition of weak isomorphisms, so it is a weak-isomorphism as well. Equation \eqref{eqcom} is an easy induction. Finally, we have
$$
v^{rs}\ast D_s=v^{r,s-1}(e_{1,p_{s-1}}^{(s-1)}\otimes1)v^{s-1,s}\ast D_s=\left(v^{r,s-1}\ast D_{s-1}\right)(e_{1,p_{s-1}}^{(s-1)}\otimes1)(v^{s-1,s}\ast D_s).
$$
We assume, by induction, that $v^{r,s-1}\ast D_{s-1}=V^{(r,s-1)}_{p_r,1}$. Also, clearly
$$
V^{(r,s-1)}_{p_r,1}(e_{1,p_{s-1}}^{(s-1)}\otimes1)V^{(s-1,s)}_{p_{s-1},1}\ne0.
$$
In particular, $v^{rs}\ne0$.

Now, for each $i=2,3,\ldots,t$, let $u_i=\left(g_1^{(i)}\right)^{-1}(\deg v^{1i})^{-1}g_{p_1}^{(1)}$, and let $u_1=1$. Define
$$
\eta=(g_1^{(1)},\ldots,g_{p_1}^{(1)},g_1^{(2)}u_2,\ldots,g_{p_2}^{(2)}u_2,\ldots,g_1^{(t)}u_t,\ldots,g_{p_t}^{(t)}u_t).
$$

\begin{Lemma}\label{reflemma}
If $J$ is graded, then the conjecture of Valenti-Zaicev is valid.
\end{Lemma}
\begin{proof}
Using the notation above, let $\mathcal{A}=(UT(n_1,\ldots,n_t),D_1,\eta)$. We note that $\mathcal{A}$ has the vector space decomposition $\bigoplus_{i,j}M_{ij}$, the same as $U$. So, for any element of the kind $e_{ij}\otimes d\in\mathcal{A}$, there exist unique $k,\ell$ such that $e_{ij}\otimes d\in M_{k\ell}$. If $k=\ell$, then $\bar{\i}$ and $\bar{\j}$ will designate the integers such that $e_{ij}\otimes d=e_{\bar{\i}\bar{\j}}^{(\ell)}\otimes d'\in M_{\ell\ell}$, for some (not necessarily homogeneous) $d'\in D_\ell$. This is well defined, since $D_1$ and each $D_r$ corresponds to same size square matrices.

Define $\psi:\mathcal{A}\to U$ by
$$
\psi(e_{ij}\otimes d)=\left\{\begin{array}{l}%
e_{\bar{\i}\bar{\j}}^{(\ell)}\otimes\psi_{1\ell}(d),\text{ if $e_{ij}\in M_{\ell\ell}$, for some $\ell$},\\%
(e_{\bar{\i}p_k}^{(k)}\otimes1)(\psi_{1k}(d)\ast v^{k\ell})(e_{1\bar{\j}}^{(\ell)}\otimes1),\text{ if $e_{ij}\in M_{k\ell}$}%
\end{array}\right.
$$
\noindent\textbf{Claim.} $\psi$ is an algebra homomorphism.

Indeed, we see that if $j\ne k$, then $\psi(e_{ij}\otimes d_1)\psi(e_{jk}\otimes d_2)=0$. So, let $e_{ij}\otimes d_1,e_{jk}\otimes d_2\in\mathcal{A}$.
\begin{enumerate}
\renewcommand{\labelenumi}{(\roman{enumi})}
\item If $e_{ij}\otimes d_1\in M_{rr}$ and $e_{jk}\otimes d_2\in M_{rs}$, then
\begin{align*}
\psi(e_{ij}\otimes d_1)\psi(e_{jk}&\otimes d_2)=(e_{\bar{\i}\bar{\j}}^{(r)}\otimes\psi_{1r}(d_1))(e_{\bar{\j}p_k}^{(r)}\otimes1)(\psi_{1r}(d_2)\ast v^{rs})(e_{1\bar{k}}^{(s)}\otimes1)\\%
&=(e_{\bar{\i}p_r}^{(r)}\otimes\psi_{1r}(d_1))(e_{p_rp_r}^{(r)}\otimes\psi_{1r}(d_2))v^{rs}(e_{1\bar{k}}^{(s)}\otimes1)\\%
&=(e_{\bar{\i}p_r}^{(r)}\otimes1)(\psi_{1r}(d_1d_2)\ast v^{rs})(e_{1\bar{k}}^{(s)}\otimes1)=\psi(e_{ik}\otimes d_1d_2).
\end{align*}
\item If $e_{ij}\otimes d_1\in M_{rs}$ and $e_{jk}\otimes d_2\in M_{ss}$ then an analogous computation is valid, using the following consequence of \eqref{eqcom}:
\begin{equation*}
v^{rs}\ast\psi_{1s}(d)=v^{rs}\ast\psi_rs(\psi_{1r}(d))=\psi_{1r}(d)\ast v^{rs}.
\end{equation*}
\item Assume $e_{ij}\otimes d_1\in M_{rs}$ and $e_{jk}\otimes d_2\in M_{s\ell}$. We can write
$$
\psi(e_{ij}\otimes d_1)=(e_{\bar{\i}p_r}^{(r)}\otimes1)\psi_{1r}(d_1)\ast v^{rs}(e_{1\bar{\j}}^{(s)}\otimes1)=(e_{\bar{\i}p_r}^{(r)}\otimes\psi_{1r}(d_1))v^{rs}(e_{1\bar{\j}}^{(s)}\otimes1).
$$
Similarly, $\psi(e_{jk}\otimes d_2)=(e_{\bar{\j}p_s}^{(s)}\otimes\psi_{1s}(d_2))v^{s\ell}(e_{1\bar{k}}^{(\ell)}\otimes1)$. Then
\begin{align*}
\psi(e_{ij}\otimes &d_1)\psi(e_{jk}\otimes d_2)\\%
&=(e_{\bar{\i}p_r}^{(r)}\otimes\psi_{1r}(d_1))v^{rs}(e_{1\bar{\j}}^{(s)}\otimes1)(e_{\bar{\j}p_s}^{(s)}\otimes\psi_{1s}(d_2))v^{s\ell}(e_{1\bar{k}}^{(\ell)}\otimes1)\\%
&=(e_{\bar{\i}p_r}^{(r)}\otimes\psi_{1r}(d_1))(\psi_{1r}(d_2)\ast v^{rs})(e_{1p_s}^{(s)}\otimes1)v^{s\ell}(e_{1\bar{k}}^{(\ell)}\otimes1)\\%
&=(e_{\bar{\i}p_r}^{(r)}\otimes\psi_{1r}(d_1d_2))v^{r\ell}(e_{1\bar{k}}^{(\ell)}\otimes1)\\%
&=\psi(e_{ik}\otimes d_1d_2).
\end{align*}
\item Finally, if $e_{ij}\otimes d_1,e_{jk}\otimes d_2\in M_{rr}$, then it is easy to check that $\psi(e_{ij}\otimes d_1)\psi(e_{jk}\otimes d_2)=\psi(e_{ik}\otimes d_1d_2)$.
\end{enumerate}
So the claim is true.

\noindent\textbf{Claim.} $\psi$ is a $G$-graded map.

Let $e_{ij}\otimes d\in\mathcal{A}$ be homogeneous, and let $r,s$ be such that $e_{ij}\otimes d\in M_{rs}$, as before. We note that, by \eqref{weakiso} and by the choice of each $u_r$, we have
\begin{align*}
\deg_{D_r}\psi_{1r}(d)&=\left(g_1^{(r)}\right)^{-1}(\deg_Uv^{1r})^{-1}g_{p_1}^{(1)}(\deg_{D_1}d)\left(g_{p_1}^{(1)}\right)^{-1}(\deg_Uv^{1r})g_1^{(r)}\\%
&=u_r(\deg_{D_1}d)u_r^{-1}.
\end{align*}
Then, by definition of the grading on $\mathcal{A}$, we have
$$
\deg_\mathcal{A} e_{ij}\otimes d=g_{\bar{\i}}^{(r)}u_r(\deg_{D_1}d)u_s^{-1}\left(g_{\bar{\j}}^{(s)}\right)^{-1}.
$$
If $r=s$, then $\psi(e_{ij}\otimes d)=e_{\bar{\i}\bar{\j}}^{(r)}\otimes\psi_{1r}(d)$, and
$$
\deg_U\psi(e_{ij}\otimes d)=g_{\bar{\i}}^{(r)}(\deg_{D_r}\psi_{1r}(d))\left(g_{\bar{\j}}^{(r)}\right)^{-1}.
$$
Thus both degrees coincide.

Now, if $r<s$, then $\psi(e_{ij}\otimes d)=(e_{\bar{\i}p_r}^{(r)}\otimes\psi_{1r}(d))v^{rs}(e_{1\bar{\j}}^{(s)}\otimes1)$. So
\begin{align*}
\deg_U\psi(e_{ij}\otimes d)&=g_{\bar{\i}}^{(r)}(\deg_{D_r}\psi_{1r}(d))\left(g_{p_r}^{(r)}\right)^{-1}(\deg_Uv^{rs})g_1^{(s)}\left(g_{\bar{\j}}^{(s)}\right)^{-1}\\%
&=g_{\bar{\i}}^{(r)}u_r(\deg_{D_1}d)u_r^{-1}\left(g_{p_r}^{(r)}\right)^{-1}(\deg_Uv^{rs})g_1^{(s)}\left(g_{\bar{\j}}^{(s)}\right)^{-1}.
\end{align*}
So, we need to show that $u_r^{-1}\left(g_{p_r}^{(r)}\right)^{-1}(\deg_Uv^{rs})g_1^{(s)}=v_s^{-1}$. Since $v^{1s}=v^{1r}(e^{(r)}_{1p_r}\otimes1)v^{rs}$, we obtain $\deg v^{1s}=\deg v^{1r}g_1^{(r)}\left(g_{p_r}^{(r)}\right)^{-1}\deg v^{r}$. Thus
$$
u_r^{-1}\left(g_{p_r}^{(r)}\right)^{-1}\deg v^{rs}g_1^{(s)}=\left(g_{p_1}^{(1)}\right)^{-1}\deg v^{1s}g_1^{(s)}=u_s^{-1}.
$$
Hence, the degrees coincide again, and the proof os complete.
\end{proof}
So, if the Jacobson radical of $U$ is graded, then we obtain a nice description of the grading. In this direction, an important result is the following theorem, due to Gordienko:

\begin{Lemma}[Corollary 3.3 of \cite{Gord}]
Let $A$ be a finite-dimensional associative algebra over a field $K$ graded by any group $G$. Suppose that either $\text{char}\,K=0$ or $\text{char}\,K>\dim A$. Then the Jacobson radical $J:=J(A)$ is a graded ideal of $A$.\qed
\end{Lemma}

Combining Gordienko's Theorem and Lemma \ref{reflemma}, we obtain
\begin{Thm}\label{main_thm}
Let $G$ be any group and consider any $G$-grading on the upper block triangular matrix algebra $A=UT(n_1,n_2,\ldots,n_t)$ over a field $K$. Suppose that either $\text{char}\,K=0$ or $\text{char}\,K>\dim A$. Then there exists a $G$-graded division algebra on $D=M_n(K)$ and an upper block triangular matrix algebra $B=UT(n_1',n_2',\ldots,n_t')$ endowed with an elementary grading, such that $A\cong B\otimes D$, where the grading on $A$ is given by \eqref{eq_degree}.\qed
\end{Thm}

Note that for the particular case where $K$ is algebraically closed of characteristic zero and $G$ is abelian (finite or not), then $J$ is automatically graded (for instance, $J$ is graded by the duality between gradings and action). In this case, the classification of division gradings over matrix algebras is known (see, for example, \cite[Chapter 1]{EldKoch}). In this way, we re-obtain the result of Valenti and Zaicev \cite{VZ2011}. More precisely, we have
\begin{Cor}
Let $G$ be an abelian group, and let $K$ be an algebraically closed field of characteristic zero. Let $U=UT(n_1,n_2,\ldots,n_t)$ be endowed with any $G$-grading. 
Then there exists a subgroup $T\subset G$, a 2-cocycle $\sigma:T\times T\to K^\times$, and a block-triangular algebra $U'=UT(n_1',n_2',\ldots,n_t')$ endowed with an elementary grading (where $n_i=n_i'|T|$, for each $i$), such that $U\cong U'\otimes K^\sigma T$.
\end{Cor}
\begin{proof}
In this case, a graded division algebra on a matrix algebra is $K^\sigma T$ (for instance, see Theorem 2.15 of \cite{EldKoch}).
\end{proof}
A natural question is if Theorem \ref{main_thm} is true without any restriction on the characteristic of the base field.

\subsection*{Acknowledgments}
The author is grateful to the Referee whose comments improved the exposition of the manuscript and corrected some proofs. This work was completed while the author was visiting Memorial University of Newfoundland (Canada) under the supervision of Professor Yuri Bahturin and Professor Mikhail Kochetov. The author thanks his doctoral advisor, Professor Plamen Koshlukov, from the State University of Campinas (Brazil).

\end{document}